\documentclass{amsart}
\usepackage{amsthm,amsmath,amsfonts,mathrsfs,xypic, amssymb, eucal,comment, enumerate}

\numberwithin{equation}{section}

\theoremstyle{plain}
\newtheorem{theorem}{Theorem}[section]
\newtheorem{proposition}[theorem]{Proposition}
\newtheorem{lemma}[theorem]{Lemma}
\newtheorem{corollary}[theorem]{Corollary}

\theoremstyle{definition}
\newtheorem{definition}[theorem]{Definition}
\newtheorem{remark}[theorem]{Remark}
\newtheorem{example}[theorem]{Example}

\title{Minimal $F$-crystals and isomorphism numbers of isosimple $F$-crystals}
\author[X. Xiao]{Xiao Xiao}
\address{Mathematics Department, Utica College, 1600 Burrstone Road, Utica, NY 13502}
\begin{document}
\begin{abstract}
In this paper we generalize minimal $p$-divisible groups defined by Oort to minimal $F$-crystals over algebraically closed fields of positive characteristic. We prove a structural theorem of minimal $F$-crystals and give an explicit formula of the Frobenius endomorphism of the basic minimal $F$-crystals that are the building blocks of the general minimal $F$-crystals. We then use minimal $F$-crystals to generalize minimal heights of $p$-divisible groups and give an upper bound of the isomorphism numbers of $F$-crystals, whose isogeny type are determined by simple $F$-isocrystals, in terms of their ranks, Hodge slopes and Newton slopes. 
\end{abstract}

\maketitle

\section{Introduction}
Let $p$ be a rational prime number and $k$ an algebraically closed field of characteristic $p$. For each Newton polygon $\nu$ of slopes in the interval $[0,1]$, Oort \cite{Oort:minimal} explicitly constructs a $p$-divisible group $H(\nu)$ of Newton polygon $\nu$ over $k$ satisfying the following property: for any $p$-divisible group $D$ over $k$, if $D[p] \cong H(\nu)[p]$, then $D \cong H(\nu)$. A $p$-divisible group $X$ over $k$ with Newton polygon $\nu$ is called \emph{minimal} if $X \cong H(\nu)$. For every $p$-divisible group $D$ over $k$, there exists a unique minimal $p$-divisible group up to isomorphism that is isogenous to $D$. A $p$-divisible group $D$ is minimal if and only if its endomorphism $\mathbb{Z}_p$-algebra $\mathrm{End}(D)$ is the maximal order in its endomorphism $\mathbb{Q}_p$-algebra $\mathrm{End}(D) \otimes_{\mathbb{Z}_p} \mathbb{Q}_p$.

Loosely speaking, minimal $p$-divisible groups are $p$-divisible groups whose isomorphism types are determined by their $p$-kernels. In order to generalize this concept to $F$-crystals, we first use the (contravariant) Dieudonn\'e module theory and call a Dieudonn\'e module over $k$ minimal if its corresponding $p$-divisible group is minimal. Put it in another way, the isomorphism class of a minimal Dieudonn\'e module $\mathcal{M}$ should be determined by $\mathcal{M}$ modulo $p$. The precise meaning of an arbitrary $F$-crystal $\mathcal{M}$ modulo $p$ or more generally modulo $p^n$ for any integer $n \geq 1$ can be defined using its isomorphism number.

To recall the definition of the isomorphism number of an $F$-crystal, we fix some notations. Let $W(k)$ be the ring of $p$-typical Witt vectors with coefficients in $k$. Let $B(k) = W(k)[1/p]$ be its field of fractions. Let $\sigma$ be the Frobenius automorphism of $W(k)$ and $B(k)$. An $F$-crystal over $k$ is a pair $\mathcal{M} = (M, \varphi)$ where $M$ is a free $W(k)$-module of finite rank and $\varphi : M \to M$ is a $\sigma$-linear monomorphism. We denote the $F$-isocrystal $(M \otimes_{W(k)} B(k), \varphi \otimes_{W(k)} \sigma)$ by $\mathcal{M}[1/p] = (M[1/p], \varphi[1/p])$, and call $\mathcal{M}$ isosimple if $\mathcal{M}[1/p]$ is simple. Unless mentioned otherwise, all $F$-crystals and $F$-isocrystals in this paper are over $k$. The isomorphism number $n_{\mathcal{M}}$ of an $F$-crystal $\mathcal{M} = (M, \varphi)$ is the smallest non-negative integer such that for every $W(k)$-linear automorphism $g$ of $M$ with the property that $g \equiv \textrm{id}_M$ modulo $p^{n_{\mathcal{M}}}$, the $F$-crystal $(M, g\varphi)$ is isomorphic to $\mathcal{M}$.

The isomorphism number of an $F$-crystal over $k$ is finite by the work of Vasiu \cite[Main Theorem A]{Vasiu:CBP}. We say that the isomorphism type of an $F$-crystal (or a Dieudonn\'e module particularly) $\mathcal{M}$ over $k$ is determined by $\mathcal{M}$ modulo $p^n$ if $n_{\mathcal{M}}  \leq n$. Using the Dieudonn\'e module theory to go back to the category of $p$-divisible groups over $k$, we know that the isomorphism type of the Dieudonn\'e module $\mathcal{DM}(D)$ of a $p$-divisible group $D$ is determined by $\mathcal{DM}(D)$ modulo $p^n$ if and only if the isomorphism type of $D$ is determined by its $p^n$-kernel $D[p^n]$.

\begin{definition} \label{definition:minimalfcrystal}
An $F$-crystal $\mathcal{M}$ over $k$ is said to be \emph{minimal} if its isomorphism number $n_{\mathcal{M}} \leq 1$.
\end{definition}

It is clear from the definition of minimal $F$-crystals that for every $F$-crystal $\mathcal{M}'$ isomorphic to a minimal $F$-crystal $\mathcal{M} = (M, \varphi)$ modulo $p$, meaning that $\mathcal{M}' \cong (M, g\varphi)$ for some $g \in \mathrm{GL}_M(W(k))$ with $g \equiv \textrm{id}_M$ modulo $p$, we have that  $\mathcal{M}'$ is isomorphic to $\mathcal{M}$. The first main result of this paper is a structural theorem of minimal $F$-crystals which is proved at the end of Section \ref{section:minimalfcrystals}.

\begin{theorem} \label{theorem:maintheorema}
For each Newton polygon $\nu$, there exists a unique minimal $F$-crystal $\mathcal{M}(\nu)$ (up to isomorphism) of Newton polygon $\nu$. Moreover, if $\Lambda$ is the finite set of distinct Newton slopes of $\mathcal{M}(\nu)$ and for each $\lambda \in \Lambda$, $m_{\lambda} \in \mathbb{Z}_{>0}$ is the multiplicity of $\lambda$, then we have a finite direct sum decomposition
\[\mathcal{M}(\nu) = \bigoplus_{\lambda \in \Lambda} \mathcal{M}_{\lambda}^{m_{\lambda}}\]
where $\mathcal{M}_{\lambda} = (M_{\lambda}, \varphi_{\lambda})$ are the unique (up to isomorphism) isosimple minimal $F$-crystals of Newton slope $\lambda$.

Furthermore, for each direct summand $\mathcal{M}_{\lambda}$, if $\textrm{rank}(M_{\lambda}) = r$, then there exists a $W(k)$-basis $\{v_1, v_2,$ $\dots, v_r\}$ of $M_{\lambda}$ such that $\varphi_{\lambda}(v_i) = p^{\lfloor i\lambda \rfloor - \lfloor (i-1)\lambda \rfloor}v_{i+1}$ for all $1 \leq i \leq r-1$ and $\varphi_{\lambda}(v_r) = p^{ r\lambda - \lfloor (r-1)\lambda \rfloor}v_{1}$.
\end{theorem}

The structure of a minimal $F$-crystal and a minimal Dieudonn\'e module of a minimal $p$-divisible group over $k$ is very similar in the following sense. First, they can be decomposed into a finite direct sum of isosimple minimal ones. Second, if the Newton slope $\lambda \in [0,1)$, then the isosimple minimal $F$-crystal $\mathcal{M}_{\lambda}$ is exactly the isosimple minimal Dieudonn\'e module corresponding to the minimal $p$-divisible group of Newton slope $\lambda$ ; if $\lambda \geq 1$, then the isosimple minimal $F$-crystal $\mathcal{M}_{\lambda}$ is isomorphic to $(M, p^{\lfloor \lambda \rfloor}\varphi_{\lambda - \lfloor \lambda \rfloor})$ where $(M, \varphi_{\lambda - \lfloor \lambda \rfloor})$ is the isosimple minimal Dieudonn\'e module of Newton slope $\lambda - \lfloor \lambda \rfloor \in [0,1)$.

We will prove that an $F$-crystal $\mathcal{M}$ is minimal if and only if its endomorphism $\mathbb{Z}_p$-algebra $\mathrm{End}(\mathcal{M})$ is the maximal order in the endomorphism $\mathbb{Q}_p$-algebra $\mathrm{End}(\mathcal{M}[1/p]) = \mathrm{End}(\mathcal{M})[1/p]$; see Proposition \ref{proposition:minimalmaxorder}. This generalizes the same property held for minimal $p$-divisible groups \cite[Section 1.1]{Oort:minimal} and thus gives another evidence that our definition of minimal $F$-crystals is the right generalization of minimal $p$-divisible groups.

In \cite{Vasiu:reconstructing}, Vasiu gives a different definition of minimal $p$-divisible group by using some technical conditions on permutations. It turns out that Vasiu's definition coincides with Oort's definition due to \cite[Theorem 1.6]{Vasiu:reconstructing}. In this paper, we also prove that our definition of minimal $F$-crystals satisfies the same technical condition as the one in Vasiu's definition; see Proposition \ref{proposition:minimalisoclinic}.

If $\mathcal{M} = (M, \varphi)$ is a Dieudonn\'e module, by fixing a $W(k)$-basis $\{v_i\}_{i=1}^r$ of $M$, there exists an element $b \in \mathrm{GL}_r(W(k))$ such that $\varphi = b\sigma$ where $\sigma$ is the $\sigma$-linear automorphism of $M$ defined by $\sigma(v_i) = v_i$ for all $1 \leq i \leq r$. In \cite{Viehmann:truncation1}, Viehmann constructed for each $b \in \mathrm{GL}_r(B(k))$, a unique minimal element $w_b\tau_{\mu_b}$ in the $\sigma$-conjugacy class $[b] = \{g^{-1}b\sigma(g) \; | \; g \in \mathrm{GL}_r(B(k))\}$ that defines the so-called minimal truncation type $(w,\mu)$ of $[b]$; see \cite[Definition 10]{Viehmann:truncation1} for their precise definitions. The $\sigma$-conjugacy class $[b]$ represents the isogeny class of $\mathcal{M}$ and the minimal truncation type of each $[b]$ defines the minimal Dieudonn\'e module within that isogeny class. Viehmann proved that $b$ is $\mathrm{GL}_r(W(k))$-$\sigma$-conjugate to the minimal element $w_b\tau_{\mu_b}$, namely, there exists a $g \in \mathrm{GL}_r(W(k))$ such that $g^{-1}b\sigma{g} = w_b\tau_{\mu_b}$, if and only if $\mathcal{M}$ is an minimal Dieudonn\'e module. This result can be generalized to minimal $F$-crystals immediately.

In the second part of this paper, we will use minimal $F$-crystals to study isomorphism numbers of isosimple $F$-crystals. Isomorphism numbers of general $F$-crystals are not easily computable based on their very definition. Therefore we would like to give (optimal) upper bounds of them in terms of other basic invariants of $F$-crystals such as ranks, Hodge slopes and Newton slopes. Optimal upper bounds of general $F$-crystals are not known but some results in certain special cases have been found. For example, Lau, Nicole and Vasiu \cite{Vasiu:traversosolved} proved that if $\mathcal{M}$ is a Dieudonn\'e module with dimension $d$ and codimension $c$, then $n_{\mathcal{M}} \leq \lfloor 2cd/(c+d) \rfloor$ and the inequality is optimal in the sense that there exists an isoclinic Dieudonn\'e module $\mathcal{M}$ such that $n_{\mathcal{M}} = \lfloor 2cd/(c+d) \rfloor$. In the case of isoclinic $F$-crystals, upper bounds that are optimal in many cases but not in general have also been found in \cite{Xiao:computing}.

For every $F$-crystal $\mathcal{M}$, let $\mathcal{M}'$ be the unique minimal $F$-crystal isogenous to $\mathcal{M}$. We define the \emph{minimal height} $q_{\mathcal{M}}$ to be the smallest non-negative integer so that there exists an isogeny $\mathcal{M}' \hookrightarrow \mathcal{M}$ whose cokernel is annihilated by $p^{q_{\mathcal{M}}}$. This is the generalization of minimal heights of $p$-divisible groups defined in \cite{Vasiu:traversosolved}. It turns out that $n_{\mathcal{M}} \leq 2q_{\mathcal{M}} +1$ just as in the case of $p$-divisible groups; see Corollary \ref{cor:estimateq}. By using the theory of minimal $F$-crystals developed in the first part of the paper, we are able to find an upper bound of $q_{\mathcal{M}}$ and thus an upper bound $n_{\mathcal{M}}$ in the case of isosimple $F$-crystals.

\begin{theorem} \label{theorem:maintheoremb}
Let $\mathcal{M}$ be an isosimple $F$-crystal with Newton slope $s/r$ in reduced form and $e$ its maximal Hodge slope. We have the following upper bound of the isomorphism number of $\mathcal{M}$:
\[n_{\mathcal{M}} \leq 2\mathrm{max}\{{\Big \lfloor}(r-\lceil\frac{s}{e}\rceil)\frac{s}{r} {\Big \rfloor}, {\Big \lfloor} \lfloor \frac{s}{e}\rfloor (e- \frac{s}{r}) {\Big \rfloor}\}+1.\]
\end{theorem}

The proof of Theorem \ref{theorem:maintheoremb} is in Section \ref{section:estimate} and it mainly relies on the search of an upper bound of the $p$-exponent of $\mathcal{M}$; see Section \ref{section:valuations} for the definition. It ultimately comes down to finding a closed formula of the Frobenius number of three natural numbers $s$, $re-s$ and $r$. In general, there is no closed formula for the Frobenius number of three or more natural numbers but our situation is a special case so that we can use a result of Brauer and Shockley (Theorem \ref{theorem:brauer}) to get a closed formula of the Frobenius number of $s$, $re-s$ and $r$. As a result, we are able to find an upper bound of the $p$-exponent of $\mathcal{M}$ and hence the isomorphism number.

Compared to the upper bound provided in \cite[Theorem 1.2]{Xiao:computing} which works for all isoclinic $F$-crystals, the upper bound provided by Theorem \ref{theorem:maintheoremb} applies only to isosimple $F$-crystals. Although somewhat limited in generality, Theorem \ref{theorem:maintheoremb} does provide a sharper upper bound than the one in \cite[Theorem 1.2]{Xiao:computing} in some cases and it can also provides optimal upper bounds as well; see Examples \ref{example1} and \ref{example3}. On the other hand, the upper bound provided in \cite[Theorem 1.2]{Xiao:computing} can be slightly better than the one in Theorem \ref{theorem:maintheoremb} in some other cases; see Example \ref{example2}.

\section{Level torsion}

The level torsion of a non-ordinary $F$-crystal is equal to its isomorphism number by \cite[Theorem 1.2]{Xiao:vasiuconjecture}, hence it can be used to compute the isomorphism number of an $F$-crystal. We would like to mention that Nie also proved this result in \cite{Nie:vasiuconj} but his definition of $F$-crystal is over the ring of formal Laurent series $k[[t]]$ instead of $W(k)$. In this section, we briefly recall this notion and present some of its properties that will be useful to study minimal $F$-crystals. We refer to \cite{Vasiu:reconstructing}, \cite{Vasiu:traversosolved}, and \cite{Xiao:vasiuconjecture} for detail expositions.

Let $\mathcal{M} = (M, \varphi)$ be an $F$-crystal and $E = \textrm{End}(M)$ be the set of all $W(k)$-linear endomorphisms of $M$. Therefore $E[1/p]$ is the set all $W(k)$-automorphisms of $M[1/p]$. We have a latticed $F$-isocrystal $(E, \varphi)$ where $\varphi : E[1/p] \to E[1/p]$ is a $\sigma$-linear isomorphism defined by the rule $\varphi(h) = \varphi h \varphi^{-1}$. Here we use $\varphi$ to denote the $\sigma$-linear endomorphism of $M$, or the $\sigma$-linear automorphism of $M[1/p]$, or the $\sigma$-linear automorphism of $E[1/p]$ by abuse of notation. By Dieudonn\'e-Manin's classification of $F$-isocrystals, there is a finite direct sum decomposition of $\mathcal{M}[1/p]$ into simple $F$-isocrystals and thus we get a direct sum decomposition
\[E[1/p] \cong L^+ \oplus L^0 \oplus L^-,\]
where $L^+$ has positive Newton slopes, $L^0$ has zero Newton slopes, and $L^-$ has negative Newton slopes. Let 
\[O^+ \subset E \cap L^+, \quad O^0 \subset E \cap L^0 , \quad O^- \subset E \cap L^-\]
be the maximal $W(k)$-submodules such that 
\[\varphi(O^+) \subset O^+, \quad \varphi(O^0) = O^0, \quad \varphi^{-1}(O^-) \subset O^-.\]
Write $O := O^+ \oplus O^0 \oplus O^-$; it is a lattice of $E[1/p]$ inside $E$, which will be called the \emph{level module} of $\mathcal{M}$. The \emph{level torsion} of $\mathcal{M}$ is the smallest non-negative integer $\ell_{\mathcal{M}}$ such that
\[p^{\ell_{\mathcal{M}}} E \subset O \subset E.\]
We can compute the level torsion of isoclinic $F$-crystals in the following way. For each integer $q>0$, let $\alpha_{\mathcal{M}}(q) \geq 0$ be the largest integer such that $\varphi^q(M) \subset p^{\alpha_{\mathcal{M}}(q)}M$ and let $\beta_{\mathcal{M}}(q) \geq 0$ be the smallest integer such that $p^{\beta_{\mathcal{M}}(q)}M \subset \varphi^{q}(M)$. Set $\delta_{\mathcal{M}}(q) := \beta_{\mathcal{M}}(q) - \alpha_{\mathcal{M}}(q) \geq 0$. It is not hard to show that for any integer $q>0$, we have $\alpha_{\mathcal{M}}(q) = q\lambda$ if and only if $\beta_{\mathcal{M}}(q) = q\lambda$. Thus either $\alpha_{\mathcal{M}}(q) = \beta_{\mathcal{M}}(q) = q\lambda$ or $q\lambda \in (\alpha_{\mathcal{M}}(q), \beta_{\mathcal{M}}(q))$.

\begin{lemma} \label{lemma:computeleveltorsion1}
Let $\mathcal{M}$ be an isoclinic $F$-crystal. Then $\ell_{\mathcal{M}} = \textrm{max}\; \{\delta_{\mathcal{M}}(q) \; | \; q \in \mathbb{Z}_{>0}\}$.
\end{lemma}
\begin{proof}
This is a generalization of \cite[Proposition 4.3(a)]{Vasiu:reconstructing} and can be proved in the same way.
\end{proof}

Let $\mathcal{M} = \oplus_{j \in J} \mathcal{M}_j$ be a finite direct sum of isoclinic $F$-crystal $\mathcal{M}_j = (M_j, \varphi_j)$ with Newton slopes $\lambda_j$. We can define $\varphi : \textrm{Hom}_{W(k)}(M_{j_1}[1/p],M_{j_2}[1/p]) \to \textrm{Hom}_{W(k)}(M_{j_1}[1/p],M_{j_2}[1/p])$ by the formula $\varphi(h) = \varphi_{j_2}h\varphi_{j_1}^{-1}$ where $h : M_{j_1}[1/p] \to M_{j_2}[1/p]$. For each pair $(j_1, j_2) \in J \times J$ with $\lambda_{j_1} \leq \lambda_{j_2}$, we define $\ell(j_1, j_2)$ to be the smallest integer such that for all $q \geq 0$, we have $p^{\ell(j_1,j_2)}\varphi^q(\textrm{Hom}_{W(k)}(M_{j_1},M_{j_2})) \subset \textrm{Hom}_{W(k)}(M_{j_1},M_{j_2})$. It is not hard to show that $\ell(j,j) = \ell_{\mathcal{M}_j}$.

\begin{lemma} \label{lemma:computeleveltorsion2}
If $\mathcal{M} = \oplus_{j \in J} \mathcal{M}_j$ is a finite direct sum of at least two isoclinic $F$-crystal as above, then we have the following equality
\[\ell_{\mathcal{M}} = \textrm{max} \; \{\ell(j_1,j_2) \; | \; (j_1, j_2) \in J \times J, \lambda_{j_1} \leq \lambda_{j_2}\},\]
and a basic estimate of the isomorphism number:
\[n_{\mathcal{M}} \leq \textrm{max}\{1, n_{\mathcal{M}_{j_1}}, n_{\mathcal{M}_{j_1}}+n_{\mathcal{M}_{j_2}}-1 \; | \; j_1, j_2 \in J, j_1 \neq j_2\}.\]
\end{lemma}
\begin{proof}
This is a generalization of \cite[Scholium 3.5.1]{Vasiu:reconstructing} and \cite[Proposition 1.4.3]{Vasiu:reconstructing}. It can be proved in the same way.
\end{proof}

\section{Minimal $F$-crystals} \label{section:minimalfcrystals}

In this section, we will generalize the concept of minimal $p$-divisible groups to minimal $F$-crystals. Unlike the definition of minimal $p$-divisible groups, which is given by an explicit construction in \cite{Oort:minimal} or \cite{Vasiu:reconstructing}, our definition of minimal $F$-crystals simply requires that the isomorphism numbers are less than or equal to $1$. After examining some basic properties of minimal $F$-crystals, we show that our definition of minimal $F$-crystals is the right generalization of minimal $p$-divisible groups.

Let $M$ be a free $W(k)$-module of rank $r$. For each pair $\mathcal{T} = (\mathcal{B}, \eta)$ where $\mathcal{B} = \{v_1, v_2, \dots, v_r\}$ is a $W(k)$-basis of $M$ and $\eta = (e_1, e_2, \dots, e_r)$ is a sequence of $r$ non-negative integers, we construct an $F$-crystal $\mathcal{M}(\mathcal{T}) = (M, \varphi_{\mathcal{T}})$ by defining $\varphi_{\mathcal{T}}(v_1) = p^{e_1}v_2, \varphi_{\mathcal{T}}(v_2) = p^{e_2}(v_3), \dots, \varphi_{\mathcal{T}}(v_r) = p^{e_r}v_1$. To ease notation, we adopt the convention that $v_{mr+i} := v_i$ for all $m \in \mathbb{Z}$. As a result, we can rewrite the definition of $\varphi_{\mathcal{T}}$ as $\varphi_{\mathcal{T}}(v_i) = p^{e_i}v_{i+1}$ for all $1 \leq i \leq r$. Clearly $e_1, e_2, \dots, e_r$ are the Hodge slopes of $\mathcal{M}(\mathcal{T})$ and $\lambda(\mathcal{T}) := \sum_{i=1}^re_i/r$ is the Newton slope of $\mathcal{M}(\mathcal{T})$. Moreover, if $0 \leq e_i \leq 1$ for all $1 \leq i \leq r$, then the $F$-crystal $\mathcal{M}(\mathcal{T})$ is called a cyclic Dieudonn\'e-Fontaine $p$-divisible object in \cite[Definition 2.2.2(a)]{Vasiu:CBP}.

Let $(M_0, \varphi_0)$ be a Dieudonn\'e module with Hodge slopes $e_1, e_2, \dots, e_r$. Using \cite[1.4 Basic Theorem C (a)]{Vasiu:modp} in our context by letting $G = \mathrm{GL}_{M_0}$, we know that for any $g_0 \in \mathrm{GL}_{M_0}(W(k))$, there exists an element $g \in \mathrm{GL}_{M_0}(W(k))$ with the property that $g \equiv \mathrm{id}_{M_0}$ modulo $p$, a $W(k)$-basis $\mathcal{B}_0 = \{v_1, v_2, \dots, v_r\}$ of $M_0$, and a permutation $\pi$ on the set $I=\{1,2,\dots,r\}$ that defines a $W(k)$-linear monomorphism $g_{\pi} :M_0 \to M_0$ with the property  that $g_{\pi}(v_i) = p^{e_i}v_{\pi(i)}$ for all $i \in I$ such that $(M_0, g_0\varphi_0)$ is isomorphic to $(M_0, gg_{\pi}\varphi_0)$. The same result is also true for $F$-crystals.

\begin{theorem}
For any $F$-crystal $\mathcal{M} = (M, \varphi)$ with Hodge slopes $e_1, e_2, \dots, e_r$, there exist an element $g \in \textrm{GL}_M(W(k))$ with the property that $g \equiv \textrm{id}_M$ modulo $p$, an $W(k)$-basis $\mathcal{B} = \{v_1, v_2, \dots, v_r\}$ of $M$, and a permutation $\pi$ on the set $I = \{1, 2, \dots, r\}$ that defines a $\sigma$-linear monomorphism $\varphi_{\pi} : M \to M$ with the property that $\varphi_{\pi}(v_i) = p^{e_i}v_{\pi(i)}$ for all $i \in I$ such that $\mathcal{M}$ is isomorphic to $(M, g\varphi_{\pi})$.
\end{theorem}
\begin{proof}
This theorem is a direct consequence of \cite[1.8 Generalizations]{Vasiu:modp}. We give a detailed proof here using a result of \cite{Viehmann:truncation1} for the sake of completion.

Let $\mathcal{B}_1 = \{u_1, u_2, \dots, u_r\}$ and $\mathcal{B}_2 = \{w_1, w_2, \dots, w_r\}$ be two $W(k)$-bases of $M$ such that $\varphi(u_i) = p^{e_i}w_i$ for all $i \in I$. Let $g_0 : M \to M$ be the $W(k)$-linear automorphism such that $g_0(u_i) = w_i$ for all $i \in I$, $g_{\varphi}: M \to M$ be the $W(k)$-linear monomorphism such that $g_{\varphi}(u_i) = p^{e_i}u_i$ for all $i \in I$, and $\sigma: M \to M$ be the $\sigma$-linear automorphism such that $\sigma(u_i) = u_i$ for all $i \in I$ . Hence $\varphi = g_0g_{\varphi}\sigma = b\sigma$ where $b := g_og_{\varphi} \in \mathrm{GL}_M(B(k))$.

Let $T$ be the maximal split torus of $\mathrm{GL}_M$ that normalizes the $\mathbb{Z}_p$-span of $\mathcal{B}_1$. Let $N$ be the normalizer of $T$ in $\mathrm{GL}_M$ and $W = N/T$ be the Weyl group of $\mathrm{GL}_M$. Hence $g_{\varphi} \in T(B(k))$ and it is equal to $\mu(p)$ where $\mu : \mathbb{G}_m \to T$ is the cocharacter which acts on each span of $u_i$ via the $e_i$-th power of the identity character.

By \cite[Theorem 1 (1)]{Viehmann:truncation1}, there exists an element $w \in W$ (here our $w$ is the equal to $ww_0w_{0,\mu}$ in the sense of \cite{Viehmann:truncation1}) and a representative $g_w \in N(W(k))$ such that the $\mathrm{GL}_M(W(k))$-$\sigma$-conjugacy class of $b \in \mathrm{GL}_M(B(k))$ contains an element of the form $g_1g_wg_{\varphi}g_2$ where $g_1,g_2 \in \textrm{Ker}(\mathrm{GL}_M(W(k)) \to \mathrm{GL}_M(k))$. Therefore there exists an element $\tilde{g} \in \mathrm{GL}_M(W(k))$ such that 
\[(M, \varphi) = (M, b\sigma) \cong (M, \tilde{g}^{-1}b\sigma(\tilde{g})\sigma) = (M, g_1g_wg_{\varphi}g_2\sigma).\]
Furthermore, we have
\[(M, g_1g_wg_{\varphi}g_2\sigma) = (M, g_1g_wg_{\varphi}\sigma \sigma^{-1}(g_2)) \cong (M, \sigma(g_2^{-1})g_1g_wg_{\varphi}\sigma) = (M, gg_wg_{\varphi}\sigma),\]
where $g: = \sigma(g_2^{-1})g_1 \in \textrm{Ker}(\mathrm{GL}_M(W(k)) \to \mathrm{GL}_M(k))$, that is, $g \equiv \textrm{id}_M$ modulo $p$.

Let $\pi$ be a permutation on $I$ such that $g_w(\langle u_i \rangle_{W(k)}) = \langle u_{\pi(i)} \rangle_{W(k)}$ for all $i \in I$. We can find units $x_1, x_2, \dots, x_r \in W(k)$ such that $v_i := x_iu_i$ and $g_w(v_i) = v_{\pi(i)}$ for all $i\in I$ because $k$ is algebraically closed. Let $\varphi_{\pi} := g_w g_{\varphi}\sigma$, we know that $\varphi_{\pi}$ is the $\sigma$-linear monomorphism such that $\varphi_{\pi}(v_i) = p^{e_i}v_{\pi(i)}$ for all $i \in I$. Hence $(M, \varphi) \cong (M, g\varphi_{\pi})$ is in the desired form and this completes the proof of the theorem.
\end{proof}

We now assume that $\mathcal{M}$ is minimal.  Then $\mathcal{M} = (M, g\varphi_{\pi})$ is isomorphic to $(M, \varphi_{\pi})$ as $g \equiv \textrm{id}_M$ modulo $p$. As the permutation $\pi = \prod_{j \in J} \pi_j$ can be decomposed into product of disjoint cycles $\pi_j$, we get that the $F$-crystal $\mathcal{M}$ is isomorphic to $\oplus_{j \in J} (M_j, \varphi_{\pi_j})$ where the rank of $M_j$ is equal to the length of the cycle $\pi_j$. For each $j \in J$, there exists a pair $\mathcal{T}_j = (\mathcal{B}_j, \eta_j)$ such that $(M_j, \varphi_{\pi_j})$ is isomorphic to $\mathcal{M}_j(\mathcal{T}_j)$. If $\mathcal{M}_j(\mathcal{T}_j)$ is ordinary, then $n_{\mathcal{M}_j(\mathcal{T}_j)} \leq 1$; see \cite[Proposition 2.9 and Corollary 2.11]{Xiao:computing}. If $\mathcal{M}_j(\mathcal{T}_j)$ is non-ordinary, we have $\ell_{\mathcal{M}_j(\mathcal{T}_j)} \leq \ell_{\mathcal{M}}$ by Lemma \ref{lemma:computeleveltorsion2}. As $\ell_{\mathcal{M}_j(\mathcal{T}_j)} = n_{\mathcal{M}_j(\mathcal{T})}$ and $\ell_{\mathcal{M}} = n_{\mathcal{M}}$ by \cite[Theorem 1.2]{Xiao:vasiuconjecture}, we have $n_{\mathcal{M}_j(\mathcal{T}_j)} \leq n_{\mathcal{M}} \leq 1$ . Thus each direct summand $\mathcal{M}_j(\mathcal{T}_j)$ of $\mathcal{M}$ is also minimal. 

We now give a concrete description of minimal $F$-crystals of the form $\mathcal{M}(\mathcal{T})$.

\begin{proposition} \label{proposition:minimalisoclinic}
Let $\mathcal{T} = (\mathcal{B}, \eta)$ be a pair, where $\mathcal{B} = \{v_1, v_2, \dots, v_r\}$ is a $W(k)$-basis of $M$ and $\eta = (e_1, e_2, \dots, e_r)$ is a sequence of nonnegative integers. Set $\lambda(\mathcal{T}) := \sum_{i=1}^r e_i /r$. The $F$-crystal $\mathcal{M}(\mathcal{T}) = (M, \varphi_{\mathcal{T}})$ is minimal if and only if $\varphi_{\mathcal{T}}$ satisfies the following condition:
\[\textrm{for all} \quad  1 \leq i, q \leq r, \quad \textrm{we have} \quad \varphi_{\mathcal{T}}^q(v_i) = p^{\lfloor q\lambda(\mathcal{T})\rfloor + \epsilon_q(i)}v_{i+q} \quad \textrm{where} \quad \epsilon_q(i) \in \{0,1\}. \;  \tag{*}\]
\end{proposition}
\begin{proof}
Suppose that $\mathcal{M}(\mathcal{T})$ satisfies the condition (*), then for all $1 \leq q \leq r$,
\[p^{\lfloor q\lambda(\mathcal{T}) \rfloor+1}M \subset \varphi_{\mathcal{T}}^q(M) \subset p^{\lfloor q\lambda(\mathcal{T}) \rfloor}M.\]
Thus $\delta_{\mathcal{M}(\mathcal{T})}(q) \leq 1$ for all $q \geq 1$. By Lemma \ref{lemma:computeleveltorsion1}, we know that $\ell_{\mathcal{M}(\mathcal{T})} \leq 1$. If $\mathcal{M}(\mathcal{T})$ is ordinary, then $n_{\mathcal{M}(\mathcal{T})} \leq 1$; if $\mathcal{M}(\mathcal{T})$ is non-ordinary, then $n_{\mathcal{M}(\mathcal{T})} = \ell_{\mathcal{M}(\mathcal{T})} \leq 1$. Thus $\mathcal{M}(\mathcal{T})$ is minimal.

Suppose now that $\mathcal{M}({\mathcal{T}})$ is minimal. If $\mathcal{M}({\mathcal{T}})$ is ordinary, then there exists a $W(k)$-basis $\{w_i\}_{i=1}^r$ of $M$ such that $\varphi_{\mathcal{T}}(w_i) = p^{e_i}w_i$ for all $1 \leq i \leq r$. One can compute that in this case the level module $O = E$. Thus $\ell_{\mathcal{M}} = 0$. If $\mathcal{M}({\mathcal{T}})$ is non-ordinary, then $\ell_{\mathcal{M}({\mathcal{T}})} = n_{\mathcal{M}({\mathcal{T}})} \leq 1$. As $\mathcal{M}(\mathcal{T})$ is isoclinic, we know that $\delta_{\mathcal{M}(\mathcal{T})}(q) \leq 1$ for all $q \geq 1$. There are two cases to be considered depending on $\delta_{\mathcal{M}(\mathcal{T})}(q)$ being $0$ or $1$:

Case (1): If $\delta_{\mathcal{M}(\mathcal{T})}(q) = 0$, that is, $\alpha_{\mathcal{M}(\mathcal{T})}(q) = \beta_{\mathcal{M}(\mathcal{T})}(q)$, then they both equal to $q\lambda(\mathcal{T})$. Thus $\varphi_{\mathcal{T}}^q(M) = p^{q\lambda(\mathcal{T})}M$, which clearly satisfies condition (*).

Case (2): If $\delta_{\mathcal{M}(\mathcal{T})}(q) = 1$, that is, $\alpha_{\mathcal{M}(\mathcal{T})}(q)+1 = \beta_{\mathcal{M}(\mathcal{T})}(q)$, then $q\lambda(\mathcal{T}) \in (\alpha_{\mathcal{M}(\mathcal{T})}(q), \beta_{\mathcal{M}(\mathcal{T})}(q))$. Hence $\alpha_{\mathcal{M}(\mathcal{T})}(q) = \lfloor q\lambda(\mathcal{T}) \rfloor$ and $\beta_{\mathcal{M}(\mathcal{T})}(q) = \lfloor q\lambda(\mathcal{T}) \rfloor+1$. Again the condition (*) holds in this case.
\end{proof}

We can further decompose isoclinic minimal $F$-crystals into isosimple minimal $F$-crystals.

\begin{proposition} \label{proposition:minimalisosimpledecompose}
Every isoclinic minimal $F$-crystal is a direct sum of isosimple minimal $F$-crystals.
\end{proposition}
\begin{proof}
Let $\mathcal{M}$ be an isoclinic minimal $F$-crystal. By the argument before Proposition \ref{proposition:minimalisoclinic}, we can assume that $\mathcal{M} = \mathcal{M}(\mathcal{T})$ for some pair $\mathcal{T} = (\mathcal{B}, \eta)$ where $\mathcal{B} = \{v_1, v_2, \dots, v_r\}$ and $\eta = (e_1, e_2, \dots, e_r)$. Consider the Newton slope $\lambda (\mathcal{T})= \sum_{i=1}^r e_i / r = s/ r'$ where $(s,r')=1$. Let $n = r/r'$. To ease notation, we adopt the convention that the index $i$ in $e_i$ and $\epsilon_q(i)$ are taken modulo $r$.

We first show that the $r$-tuple $(e_1, e_2, \dots, e_r)$ is $r'$-periodic, that is, $e_i = e_{i+r'}$ for all $1 \leq i \leq r$. For each $1 \leq i \leq r$, let $d_i = \sum_{j=i}^{i+r'-1}e_j$. By Proposition \ref{proposition:minimalisoclinic}, we know that $d_i = \lfloor r'\lambda(\mathcal{T}) \rfloor + \epsilon_{r'}(i) =s+\epsilon_{r'}(i)$. As $\sum_{i=1}^r d_i = r'\sum_{i=1}^re_i = r'r\lambda(\mathcal{T}) = rs$, we get that $\epsilon_{r'}(i) = 0$ for all $1 \leq i \leq r$, that is, $d_1 = d_2 = \cdots = d_r$. This implies that $e_i = e_{i+r'}$ for all $1 \leq i \leq r$.

To show that $(M, \varphi_{\mathcal{T}})$ can be decomposed into a direct sum of isosimple minimal $F$-crystals, we first observe that for any $x \in W(\mathbb{F}_{p^r})\; \backslash \; pW(\mathbb{F}_{p^r})$, the following set
\[\mathcal{S}(x) = \{z_j(x) := \sum_{i=0}^{n-1} \sigma^{ir'+j-1}(x)v_{ir'+j} \; | \; j=1,2,\dots,r' \}\]
of the elements of $M$ satisfies the following two properties:
\begin{enumerate}
\item the set $\mathcal{S}(x)$ is linearly independent over $W(k)$;
\item for $1 \leq j \leq r'$, $\varphi(z_j(x)) = p^{e_j}z_{j+1}(x)$ and thus $\varphi^{r'}(z_j(x)) = p^sz_j(x)$.
\end{enumerate}
Therefore the set $\mathcal{S}(x)$ generates an isosimple $F$-crystal $\mathcal{M}_{\mathcal{S}(x)}$ of slope $s/r'$.

We now show that there exist $x_1, x_2, \dots, x_{n} \in W(\mathbb{F}_{p^r})\; \backslash \; pW(\mathbb{F}_{p^r})$ such that the set $\{z_j(x_i)\}_{i=1}^n$ will generate $W(k)$-submodules of $M$ that are also direct summand of $M$ for all $1 \leq j \leq r'$.

To see that we can find such $x_1, x_2, \dots, x_{n}$, it is enough to show that the determinant
\[
\mathrm{det}
\begin{pmatrix}
x_1 & \sigma^{r'}(x_1) & \sigma^{2r'}(x_1) & \cdots & \sigma^{(n-1)r'}(x_1) \\
x_2 & \sigma^{r'}(x_2) & \sigma^{2r'}(x_2) & \cdots & \sigma^{(n-1)r'}(x_2) \\
\vdots & \vdots & \vdots & \ddots & \vdots \\
x_{n'} & \sigma^{r'}(x_{n}) & \sigma^{2r'}(x_{n}) & \cdots & \sigma^{(n-1)r'}(x_{n}) \\
\end{pmatrix}
\]
is invertible in $W(k)$. Let $y_i$ be the image of $x_i$ under the projection map $W(\mathbb{F}_{p^r}) \to \mathbb{F}_{p^r}$. It suffices to show that the determinant
\[
\mathrm{det}
\begin{pmatrix}
y_1 & y_1^{p^{r'}} & y_1^{p^{2r'}} & \cdots & y_1^{p^{(n-1)r'}} \\
y_2 & y_2^{p^{r'}} & y_2^{p^{2r'}} & \cdots & y_2^{p^{(n-1)r'}} \\
\vdots & \vdots & \vdots & \ddots & \vdots \\
y_{n} & y_{n}^{p^{r'}} & y_{n}^{p^{2r'}} & \cdots & y_{n}^{p^{(n-1)r'}} \\
\end{pmatrix} \neq 0
\]
in $\mathbb{F}_{p^r}$. The determinant is non-zero if and only if $y_1, y_2, \dots, y_{n'} \in \mathbb{F}_{p^r}$ are linearly independent over $\mathbb{F}_{p^{r'}}$ (see \cite[Lemma 1.3.3]{GOSS}). We can make such choices as $\mathrm{dim}_{\mathbb{F}_{p^{r'}}}(\mathbb{F}_{p^r}) = n$ (but the choices are not natural). Hence the set $\bigcup_{i=1}^{n} S(x_i)$ is a $W(k)$-basis of $M$ and each subset $S(x_i)$ generates an isosimple $F$-crystal $\mathcal{M}_{\mathcal{S}(x_i)}$ of slope $s/r'$ that is also a direct summand of $M$. Therefore $\mathcal{M}(\mathcal{T})$ is a direct sum of isosimple $F$-crystals $\bigoplus_{i=1}^{r'} \mathcal{M}_{\mathcal{S}(x_i)}$. By Lemma \ref{lemma:computeleveltorsion2}, each direct summand $\mathcal{M}_{\mathcal{S}(x_i)}$ is minimal as well.
\end{proof}

Given a non-negative rational number $\lambda = s/r$ in reduced form, we can construct an $F$-crystal $(M_{\lambda}, \varphi_{\mathcal{T_{\lambda}}})$ with Newton slope $\lambda$ in the following way. Let $M_{\lambda}$ be a free $W(k)$-module of rank $r$ and $\mathcal{T}_{\lambda} = (\mathcal{B}, \eta_{\lambda})$ where $\mathcal{B} = \{v_1, v_2, \dots, v_r\}$ is a $W(k)$-basis of $M$ and $\eta_{\lambda} = (\lfloor \lambda \rfloor, \lfloor 2\lambda \rfloor - \lfloor \lambda \rfloor, \dots, \lfloor r\lambda \rfloor - \lfloor (r-1)\lambda \rfloor)$.

\begin{proposition} \label{proposition:isosimpleminimalexist}
The $F$-crystal $(M_{\lambda}, \varphi_{\mathcal{T}_{\lambda}})$ constructed above is an isosimple minimal $F$-crystal of Newton slope $\lambda$.
\end{proposition}
\begin{proof}
As $\sum_{i=1}^r (\lfloor i\lambda \rfloor - \lfloor (i-1) \lambda \rfloor ) = r \lambda=s$, the $F$-isocrystal $(M_{\lambda}[1/p], \varphi_{\mathcal{T}_{\lambda}}[1/p])$ is simple with Newton slope $\lambda$. Hence it suffices to show that it is minimal.

By Proposition \ref{proposition:minimalisoclinic}, we have to show that for all $1 \leq i,q \leq r$, $\varphi_{\mathcal{T}_{\lambda}}^q(v_i) = p^{\lfloor q \lambda \rfloor + \epsilon_q(i)}v_{i+q}$ where $\epsilon_q(i) \in \{0,1\}$. If $i + q \leq r$, then 
\[\varphi_{\mathcal{T}_{\lambda}}^q(v_i) = p^{\lfloor (i+q-1)\lambda \rfloor - \lfloor (i-1)\lambda \rfloor}v_{i+q}.\]
If $i+q >r$, then 
\[\varphi_{\mathcal{T}_{\lambda}}^q(v_i) = p^{\lfloor (i+q-r-1)\lambda \rfloor + r \lambda - \lfloor (i-1)\lambda \rfloor}v_{i+q-r}.\]
Thus $\epsilon_q(i) = \lfloor (i-1)\lambda + q\lambda \rfloor - \lfloor (i-1)\lambda \rfloor - \lfloor q\lambda \rfloor \in \{0,1\}$ in both cases.
\end{proof}

\begin{proposition} \label{proposition:isosimpleunique}
For every rational number $\lambda \geq 0$, there exists a unique isosimple minimal $F$-crystal of Newton slope $\lambda$ up to isomorphism. 
\end{proposition}
\begin{proof}
For every rational number $\lambda \geq 0$, the $F$-crystal $(M_{\lambda}, \varphi_{\mathcal{T}_{\lambda}})$ is minimal of Newton slope $\lambda$ by Proposition \ref{proposition:isosimpleminimalexist}. Hence it suffices to prove the uniqueness. 

Let $\mathcal{T} = (\mathcal{B},\eta)$ where $\mathcal{B} = \{w_1,w_2,\dots,w_r\}$ is a $W(k)$-basis of $M$ and $\eta = (e_1,e_2,\dots,e_r)$. Suppose that $\mathcal{M}(\mathcal{T}) = (M, \varphi_{\mathcal{T}})$ is an isosimple minimal $F$-crystal with Newton slope $\lambda = s/r$ in reduced form.

By Proposition \ref{proposition:minimalisoclinic}, for each $1 \leq i \leq r$, we know that $e_i$ is equal to $\lfloor \lambda \rfloor$ or $\lceil \lambda \rceil$. If $s = \lfloor \lambda \rfloor r+s'$ where $ 0 \leq s' < r$, then we have $(r-s')\lfloor \lambda \rfloor + s' \lceil \lambda \rceil = s$. We have $(M, \varphi_{\mathcal{T}}) = (M, p^{\lfloor \lambda \rfloor}\varphi_{\mathcal{T}'})$ where $\mathcal{T}' = (\mathcal{B},\eta')$, $\eta' = (e_1-\lfloor \lambda \rfloor, e_2 - \lfloor \lambda \rfloor, \dots, e_r - \lfloor \lambda \rfloor)$. It is clear that $(M, \varphi_{\mathcal{T}'})$ is an isosimple minimal Dieudonn\'e module. Thus we can assume that $(M, \varphi_{\mathcal{T}})$ is a minimal Dieudonn\'e module of a minimal $p$-divisible group to begin with.

By \cite[1.6 Main Theorem B]{Vasiu:reconstructing}, a $p$-divisible group $D$ is minimal if and only $n_D \leq 1$. Hence for each minimal Dieudonn\'e module $(M, \varphi_{\mathcal{T}})$ of a minimal $p$-divisible group with Newton polygon $\nu$, we know that $(M, \varphi_{\mathcal{T}})$ is isomorphic to the minimal Dieudonn\'e module of the minimal $p$-divisible group $H(\nu)$. This shows that minimal $p$-divisible groups exist uniquely within each isogeny class. Thus minimal $F$-crystals exist uniquely within each isogeny class. We complete the proof of the lemma.
\end{proof}

\begin{proposition} \label{proposition:directsumminimal}
If $\mathcal{M}_1$ and $\mathcal{M}_2$ are two minimal $F$-crystals over $k$, then $\mathcal{M}_1 \oplus \mathcal{M}_2$ is also minimal.
\end{proposition}
\begin{proof}
By Proposition \ref{proposition:minimalisosimpledecompose}, both $\mathcal{M}_1$ and $\mathcal{M}_2$ are direct sums of isoclinic minimal $F$-crystals. Based on this, the proposition follows from Definition \ref{definition:minimalfcrystal} and the Lemma \ref{lemma:computeleveltorsion2}.
\end{proof}

We give another description of isosimple minimal $F$-crystals which will be useful in the next section. Let $\lambda=s/r$ be in the reduced form. There exist $m \in \mathbb{Z}$, $n \in \mathbb{Z}_{\geq 0}$ such that $mr-ns=1$. Let $K_{\lambda}$ be the associative $B(\mathbb{F}_{p^{r}})$-algebra with identity generated by $\xi$ such that $\xi^r=p$ and $x\xi = \xi \sigma^{-n}(x)$ for every $x \in B(\mathbb{F}_{p^{r}})$. The definition of $K_{\lambda}$ is independent of the choices of $m$ and $n$. Note that $K_{\lambda} \cong K_{\lambda+l}$ for all $l \in \mathbb{Z}$ as associative $B(\mathbb{F}_{p^{r}})$-algebras. It is well-known that $K_{\lambda}$ is a central simple division algebra over $\mathbb{Q}_p$. For $\xi^i \otimes x \in K_{\lambda} \otimes_{B(\mathbb{F}_{p^{r}})} B(k)$, define $\varphi(\xi^i \otimes x) = \xi^{i+s} \otimes \sigma(x)$ and thus $(K_{\lambda} \otimes_{B(\mathbb{F}_{p^{r}})} B(k), \varphi)$ is an $F$-isocrystal. It is not hard to show that $(K_{\lambda} \otimes_{B(\mathbb{F}_{p^{r}})} B(k), \varphi)$ is isomorphic to the simple $F$-isocrystal $\mathcal{N}_{\lambda}$ of Newton slope $\lambda$ and that all of its endomorphisms are the left multiplication by elements of $K_{\lambda}$. Similarly, let $L_{\lambda}$ be the associative $W(\mathbb{F}_{p^{r}})$-algebra with identity generated by $\xi$ with $\xi^r = p$ and $x\xi = \xi \sigma^{-n}(x) $ for any $x \in W(\mathbb{F}_{p^{r}})$. Then $\mathcal{M}_{\lambda} :=  L_{\lambda} \otimes_{W(\mathbb{F}_{p^{r}})} W(k)$ is an $F$-crystal of rank $r$, where the Frobenius endomorphism is defined by $\varphi(\xi^i \otimes x) = \xi^{i+s} \otimes \sigma(x)$ for any $\xi^i \otimes x \in \mathcal{M}_{\lambda}$. It is clear that $\mathcal{M}_{\lambda}$ is isosimple with Newton slope $\lambda$. As $\mathcal{M}_{\lambda}[1/p]=(K_{\lambda} \otimes_{B(\mathbb{F}_{p^r})} B(k), \varphi)$, every endomorphism of $\mathcal{M}_{\lambda}$ is a left multiplication by some element of $L_{\lambda}$. Hence $L_{\lambda}$ is the endomorphism $\mathbb{Z}_p$-algebra of $\mathcal{M}_{\lambda}$.

\begin{proposition}
The $F$-crystal $\mathcal{M}_{\lambda}$ is isosimple minimal.
\end{proposition}
\begin{proof}
By Proposition \ref{proposition:minimalisoclinic}, it suffices to check that for any $\xi^i \otimes 1\in \mathcal{M}_{\lambda}$, $0 \leq i \leq r-1$, and any $q \in \mathbb{Z}_{\geq 0}$,
\[\varphi^q(\xi^i \otimes 1) \in p^{\lfloor qs/r \rfloor + \epsilon_q(i)}\mathcal{M}_{\lambda} \quad \textrm{where} \quad \epsilon_q(i) \in \{0,1\}.\]
This is equivalent to check that $\lfloor (i+qs)/r \rfloor - \lfloor qs/r \rfloor \in \{0,1\}$, which is obvious.
\end{proof}

\begin{proposition} \label{proposition:minimalmaxorder}
An $F$-crystal $\mathcal{M}$ is minimal if and only if its endomorphism $\mathbb{Z}_p$-algebra $\mathrm{End}(\mathcal{M})$ is a maximal order of its endomorphism $\mathbb{Q}_p$-algebra $\mathrm{End}(\mathcal{M}[1/p])$.
\end{proposition}
\begin{proof}
Suppose $\mathcal{M}$ is minimal, then it can be decomposed into a finite direct sum of isoclinic minimal $F$-crystals $\oplus_{i \in I} \mathcal{M}^{m_i}_{\lambda_i}$ where $\mathcal{M}_{\lambda_i}$ are the isosimple minimal $F$-crystals of Newton slope $\lambda_i$ and $m_i$ are the multiplicities. Hence $\mathrm{End}(\mathcal{M})$ is a product of matrix rings $M_{m_i}(L_{\lambda})$. Similarly the endomorphism $\mathbb{Q}_p$-algebra $\mathrm{End}(\mathcal{M}[1/p])$ is a product of matrix algebras  $M_{m_i}(K_{\lambda})$. As $L_{\lambda}$ is the maximal order of $K_{\lambda}$, we know that each $M_{m_i}(L_{\lambda})$ is a maximal order of $M_{m_i}(K_{\lambda})$ by \cite[Theorem 8.7]{MaxOrd}.

Suppose now $\mathrm{End}(\mathcal{M})$ is a maximal order of $\mathrm{End}(\mathcal{M}[1/p])$. As $\mathrm{End}(\mathcal{M}[1/p])$ is a product of matrix algebras $M_{m_i}(K_{\lambda})$, we know that $\mathrm{End}(\mathcal{M})$ is a product of matrix rings $M_{m_i}(L_{\lambda})$. There exist finite direct sum decomposition of $\mathcal{M}[1/p]  = \oplus_{i \in I} \mathcal{N}_i$ and $\mathcal{M}  = \oplus_{i \in I} \mathcal{M}_i$ so that $\mathcal{M}_i[1/p] = \mathcal{N}_i$, $\mathrm{End}(\mathcal{M}_i) \cong L_{\lambda}$ and $\mathrm{End}(\mathcal{N}_i[1/p]) \cong K_{\lambda}$. Hence we can assume that $\mathrm{End}(\mathcal{M}[1/p]) = K_{\lambda}$ and $\mathrm{End}(\mathcal{M}) = L_{\lambda}$ to begin with. We use the same representation of $L_{\lambda}$ as above. Write $\lambda = s/r$ in reduced form. Let $\widetilde{M} = \{x \in M \; | \; \varphi^r(x) = p^sx\}$ and $\widetilde{N_{\lambda}} = \{x \in N_{\lambda} \; | \; \varphi^r(x) = p^sx\}$ where $N_{\lambda} = M[1/p]$. Clearly $\widetilde{M} = \widetilde{N_{\lambda}} \cap M$. For any $x \in M$, $x \in \widetilde{M}$ if and only if $\xi(x) \in \widetilde{M}$ as $\xi$ is injective. Therefore $\widetilde{M} \not\subset \xi(M)$. Let $\alpha \in \widetilde{M} \; \backslash \; \xi(M)$. Then $\alpha, \xi(\alpha), \dots, \xi^{r-1}(\alpha)$ is a basis of $M$ over $W(k)$. As $\widetilde{M}/p\widetilde{M} \cong \widetilde{M}/\xi^r(\widetilde{M})$ is a $r$-dimensional vector space over $B(\mathbb{F}_{p^r})$ and $\widetilde{M}/\xi(\widetilde{M}) \cong \xi^i(\widetilde{M})/\xi^{i+1}(\widetilde{M})$ for $i \geq 1$, we know that $\xi^i(\widetilde{M})/\xi^{i+1}(\widetilde{M})$ has dimension one. Thus there exists $a \in \mathbb{Z}$ and $\epsilon \in W(\mathbb{F}_{p^r})^{\times}$ such that $\varphi(x) = \epsilon\xi^a(x)$. As $\varphi^r(x) = p^sx$, we get that $a = s$ and the norm $N_{W(\mathbb{F}_{p^r})/\mathbb{Z}_p}(\epsilon) = 1$. By Hilbert's 90, we can replace $x$ with $x'$ by a unit of $W(\mathbb{F}_{p^r})$ so that $\varphi(x') = \xi^s(x')$. One can show that the map $x' \mapsto \xi \otimes 1$ defines an isomorphism between $\mathcal{M}$ and $\mathcal{M}_{\lambda}$.
\end{proof}

Proposition \ref{proposition:minimalmaxorder} shows that our definition of minimal $F$-crystals is the right generalization of minimal $p$-divisible groups; see \cite[Section 1.1]{Oort:minimal}.

\begin{proof}[Proof of Theorem \ref{theorem:maintheorema}]
By the paragraph before Proposition \ref{proposition:minimalisoclinic} and Proposition \ref{proposition:minimalisosimpledecompose}, there is a finite direct sum decomposition of $\mathcal{M}(\nu)$ into isosimple minimal $F$-crystals $(M_{\lambda}, \varphi_{\mathcal{T}_{\lambda}})$ constructed before Proposition \ref{proposition:isosimpleminimalexist}. Because isosimple minimal $F$-crystals are unique up to isomorphism by Proposition \ref{proposition:isosimpleunique}, we see that $\mathcal{M}_{\lambda}$ is isomorphic to $(M_{\lambda}, \varphi_{\mathcal{T}_{\lambda}})$. The uniqueness up to isomorphism of $\mathcal{M}_{\lambda}$ within each isogeny class also follows from Proposition \ref{proposition:isosimpleunique}.
\end{proof}

\section{Valuations on $F$-crystals} \label{section:valuations}

In this section, we briefly recall some basic theory of valuations on $F$-crystals following \cite{Vasiu:traversosolved}. Then we use it to give another equivalent description of minimal $F$-crystals.

\begin{definition}
A valuation on a $W(k)$-module $M$ is a set function $w: M \to \mathbb{R} \cup \{ \infty \}$ such that the following two properties hold:
\begin{enumerate}
\item $w(ax) = \textrm{ord}_p(a) + w(x)$ for all $a \in W(k)$ and $x \in M$;
\item $w(x + y) \geq \textrm{min}\{w(x), w(y)\}$ for all $x, y \in M$. 
\end{enumerate}
\end{definition}

As $w(0) = w(p \cdot 0) = \mathrm{ord}_p(p) + w(0) = 1 + w(0)$, we have $w(0) = \infty$. Similarly, if $x \in M$ is a torsion element we have $w(x) = \infty$. This means that $w$ factors through $M / M_\textrm{tor}$ where $M_\textrm{tor}$ is the torsion submodule of $M$. A valuation is called \emph{non-degenerate} if $w(x) = \infty$ implies that $x = 0$. It is called \emph{non-trivial} if $w(x) \neq \infty$ for some $x \in M$. It is easy to see that a valuation on a $W(k)$-module $M$ extends uniquely to the $B(k)$-vector space $M[1/p]$.

\begin{definition}
An $F$-valuation of slope $\lambda \in \mathbb{R}$ on an $F$-crystal $(M, \varphi)$ is a valuation $w$ on $M$ such that $w(\varphi(x)) = w(x) + \lambda$ for all $x \in M$. It clearly extends to the $F$-isocrystal $(M[1/p], \varphi[1/p])$.
\end{definition}

Let $B(k)\{T, T^{-1}\}$ be the noncommutative Laurent polynomial ring and let $\mathbb{D} = B(k)\{T, T^{-1}\}/I$ where $I$ is the two-sided ideal generated by all elements of the form $Ta-\sigma(a)T$ for all $a \in B(k)$. For any $\sum_{i \in \mathbb{Z}} a_i T^i \in \mathbb{D}$, the following rule 
\[w_{\lambda}(\sum_{i \in \mathbb{Z}} a_i T^i) = \textrm{min}\{\textrm{ord}_p(a_i) + i\lambda \; | \; i \in \mathbb{Z} \}\]
defines an $F$-valuation of slope $\lambda$ on $\mathbb{D}$. Each $F$-isocrystal is a $\mathbb{D}$-module where $T$ acts as the Frobenius automorphism.

\begin{proposition} \label{prop:valuationoffbyconstant}
Let $\mathcal{N}$ be an $F$-isocrystal. There exists a non-degenerate $F$-valuation of slope $\lambda$ on $\mathcal{N}$ if and only if $\mathcal{N}$ is isoclinic of slope $\lambda$. If $\mathcal{N}$ is simple of slope $\lambda$, then any two non-trivial $F$-valuations of slope $\lambda$ on $\mathcal{N}$ differ by a constant.
\end{proposition}
\begin{proof}
See \cite[Lemma 5.3]{Vasiu:traversosolved}.
\end{proof}

\begin{proposition} \label{prop:valuationminimal}
Let $\mathcal{M}$ be an $F$-crystal. Let $\mathscr{W}$ be the set of all $F$-valuations $w$ of slope $\lambda$ on $\mathcal{M}[1/p]$ such that $w(x) \geq 0$ for all $x \in M$. Then $\mathscr{W}$ has a minimal element $w_0$, that is $w_0(x) \leq w(x)$ for all $x \in M[1/p]$ and for all $w \in \mathscr{W}$. The valuation $w_0$ is non-degenerate if and only if $\mathcal{M}$ is isoclinic of slope $\lambda$.
\end{proposition}
\begin{proof}
See \cite[Lemma 5.4]{Vasiu:traversosolved}.
\end{proof}

For any $F$-valuation $w$ of slope $\lambda$ defined on an $F$-isocrystal $(N, \varphi)$, and for each real number $\alpha$, we denote
\[N^{w \geq \alpha} = \{x \in N \; | \; w(x) \geq \alpha\} \quad \textrm{and} \quad N^{w > \alpha} = \{x \in N \; | \; w(x) > \alpha\}.\]

\begin{theorem} \label{thm:anotherminimaldes}
Let $\mathcal{M}$ be an isoclinic $F$-crystal of Newton slope $\lambda$. Then the following three statements are equivalent.
\begin{enumerate}[(a)]
\item The $F$-crystal $\mathcal{M}$ is minimal.
\item For any $\Phi \in \mathbb{D}$ such that $w_{\lambda}(\Phi) \geq 0$, we have $\Phi(M) \subset M$.
\item For some $F$-valuation $w$ on $\mathcal{M}[1/p]$ of slope $\lambda$, we have $M = M[1/p]^{w \geq 0}$.
\end{enumerate}
\end{theorem}

Theorem \ref{thm:anotherminimaldes} is the generalization of \cite[Proposition 5.17]{Vasiu:traversosolved} to $F$-crystals. We first prove a lemma.

\begin{lemma} \label{lemma:anotherminimaldes}
If $\mathcal{M}$ is isosimple, then statements (a) and (c) in Theorem \ref{thm:anotherminimaldes} are equivalent.
\end{lemma}
\begin{proof}
Let $w$ be an $F$-valuation of slope $\lambda$ on $\mathcal{M}[1/p]$. For $f \in \mathrm{End}(\mathcal{M}[1/p])$, define an $F$-valuation $w'$ of slope $\lambda$ on $\mathcal{M}[1/p]$ by $w'(x) = w(fx)$. By Proposition \ref{prop:valuationoffbyconstant}, there exists $v(f) \in \mathbb{R} \cup \{\infty\}$ such that $w(fx) = w(x) + v(f)$ for all $x \in \mathcal{M}[1/p]$. Then $v$ is the unique valuation on $\textrm{End}(\mathcal{M}[1/p])$ that extends the $p$-adic valuation on $\mathbb{Q}_p$. The maximal order of $\mathrm{End}(\mathcal{M}[1/p])$ is $\mathrm{End}(\mathcal{M}[1/p])_0 = \{f \in \textrm{End}(\mathcal{M}[1/p]) \; | \; v(f) \geq 0\}$.

Suppose $M = M[1/p]^{w \geq 0}$ for some $F$-valuation $w$ of slope $\lambda$, then $\mathrm{End}(\mathcal{M})$ consists of those endomorphisms $f$ such that $v(f) \geq 0$. Suppose that there exists $f \in \mathrm{End}(\mathcal{M})$ such that $v(f)<0$, then there exists $x \in M$ with $w(x) = 0$ so that $w(f(x)) = w(x)+v(f) = v(f) <0$. Thus $f(x) \notin M$ and $f \notin \mathrm{End}(\mathcal{M})$. Therefore $\mathrm{End}(\mathcal{M})$ is exactly the maximal order of $\mathrm{End}(\mathcal{M}[1/p])$. By Proposition \ref{proposition:minimalmaxorder}, we know that $\mathcal{M}$ is minimal.

Suppose $\mathcal{M}$ is minimal of rank $r$. Let $\lambda = s/r$ be in reduced form and let $a \in \mathbb{Z}_{>0}$ and $b \in \mathbb{Z}$ be two integers such that $as+br=1$. We can choose an $F$-valuation $w$ of slope $\lambda$ on $\mathcal{M}[1/p]$ such that $\mathbb{Z} \subset w(\mathcal{M}[1/p])$. For any $n \in \mathbb{Z}$, let $x \in \mathcal{M}[1/p]$ be such that $w(x) = bn$. Then the $F$-valuation of $\varphi^{na}(x)$ is $n/r$ as follows:
\[w(\varphi^{na}(x)) = w(x)+na\lambda = bn+na\lambda = n/r.\]
Therefore $(1/r)\mathbb{Z} \subset w(\mathcal{M}[1/p])$. Since we can define an $F$-valuation $w'$ on $\mathcal{M}[1/p]$ such that $w'(\mathcal{M}[1/p]) = (1/r)\mathbb{Z}$, by the second half of Proposition \ref{prop:valuationoffbyconstant}, we know that $w(\mathcal{M}[1/p]) = (1/r)\mathbb{Z}$. Let $\pi$ be a generator of the maximal ideal of $\textrm{End}(\mathcal{M}[1/p])_0$, we have $v(\pi) = 1/r$. As $\pi(M) \subset M$, $w(M)$ is of the form $\{i/r \; | \; i \in \mathbb{Z}, i \geq i_0\}$ for some $i_0$. By replacing $w$ with $w - i_0/r$, we can assume that $i_0=0$ and hence $M = M[1/p]^{w \geq 0}$ for some $w$.
\end{proof}

\begin{proof}[Proof of Theorem \ref{thm:anotherminimaldes}]
Let $\lambda = s/r$ be the reduced form and $m$ be the multiplicity of $\mathcal{M}[1/p]$, that is, $m = \textrm{dim}_{B(k)}(M[1/p])/r$. Let $\Phi_0 = T^rp^{-s} \in \mathbb{D}$ and choose $\Phi = T^ap^b$ such that $w_{\lambda}(\Phi) = 1/r$. The element $\Phi$ is unique up to multiplication by an integral power of $\Phi_0$. 

Suppose the statement (b) is true, we show that there exist a $W(k)$-basis of $M$ of the form $\Upsilon = (\Phi^ix_j)_{0 \leq i < r, 1 \leq j \leq m}$ such that each $x_j \in M$ satisfies the equation $\Phi_0 (x_j) = x_j$. Indeed, let $\widetilde{M} = \{x \in M \; | \; \Phi_0(x) = x\}$. As $w_{\lambda}(\Phi_0) = 0$ and $\Phi_0$ preserves $M$ by (b), we know that $M \cong \widetilde{M} \otimes_{W(\mathbb{F}_{p^r})} W(k)$. As $\Phi^r : \widetilde{M} \to \widetilde{M}$ is multiplication by $p$, the quotient $\widetilde{M}/\Phi(\widetilde{M})$ is a $\mathbb{F}_{p^r}$-vector space of dimension $m$. By taking $\{x_1, x_2, \dots, x_m\} \in \widetilde{M}$ so that $\{\bar{x}_1, \bar{x}_2, \dots, \bar{x}_m\}$ is basis of $\widetilde{M}/\Phi(\widetilde{M})$, then the projection of the set $\{\Phi^i(x_1), \Phi^i(x_2), \dots, \Phi^i(x_m)\}$ is a basis of $\Phi^i(\widetilde{M})/\Phi^{i+1}(\widetilde{M})$. Therefore $\Upsilon$ is a basis of $\widetilde{M}$ and thus a basis of $M$.

We prove (b) $\Rightarrow$ (a). Let $\Upsilon$ be as above. Let $M_j = \mathbb{D}^{w_{\lambda} \geq 0}x_j$ for $1 \leq j \leq m$ and $\mathcal{M}_j = (M_j, \varphi)$. Then there is a decomposition of $\mathcal{M} = \mathcal{M}_1 \oplus \mathcal{M}_2 \oplus \cdots \oplus \mathcal{M}_m$ such that $\mathcal{M}_1 \cong \mathcal{M}_2 \cong \cdots \cong \mathcal{M}_m$. It suffices to show that $\mathcal{M}_1$ is minimal by Proposition \ref{proposition:directsumminimal}. Define an $F$-valuation $w$ of slope $\lambda$ on $M_1[1/p]$ as follows:
\[w(\sum_i a_i \Phi^i(x_1)) = \textrm{min}_i \{ \textrm{ord}_p(a_i) + i/r\}.\]
Hence $M_1 = M_1[1/p]^{w \geq 0}$. By Lemma \ref{lemma:anotherminimaldes}, $\mathcal{M}_1$ is minimal.

We prove (a) $\Rightarrow$ (c). Let $\mathcal{M} = \mathcal{M}_1 \oplus \mathcal{M}_2 \oplus \cdots \oplus \mathcal{M}_m$ be a decomposition of minimal isosimple $F$-crystals. By Lemma \ref{lemma:anotherminimaldes}, for each $1 \leq i \leq m$, there exists $w_i$ such that $M_i = M_i[1/p]^{w_i \geq 0}$. Define an $F$-valuation $w$ of slope $\lambda$ on $M$ such that $w(x) = \textrm{min}_i\{w_i(x_i) \; | \; x_i \; \textrm{is the projection of $x$ to $M_i$}\}$. Hence $M = M[1/p]^{w \geq 0}$.

We prove (c) $\Rightarrow$ (b). Let $\Phi = \sum_i  a_i T^i \in \mathbb{D}$ such that $w_{\lambda}(\Phi) \geq 0$. Then $\textrm{ord}_p(a_i) + i\lambda \geq 0$ for all $i$. For any $x \in M$, that is, $w(x) \geq 0$, we want to show that $w(\Phi(x)) \geq 0$. As $w(a_i\varphi^i(x)) = \textrm{ord}_p(a_i)+i\lambda + w(x) \geq 0$, we have $w(\Phi(x)) \geq 0$.
\end{proof}

The next proposition is proved in the case of Dieudonn\'e modules in \cite[Lemma 4.2]{Yu:finiteendabvar}.

\begin{proposition} \label{prop:minimalminimal}
Let $(M, \varphi)$ be an $F$-crystal. There exist two special minimal $F$-crystals:
\begin{enumerate}[(1)]
\item The minimal minimal $F$-crystal $(M_+, \varphi)$ inside $(M[1/p], \varphi[1/p])$ containing $(M, \varphi)$, i.e. for every minimal $F$-crystal $(M', \varphi)$ inside $(M[1/p], \varphi[1/p])$ containing $(M, \varphi)$, we have $(M_+, \varphi) \subset (M', \varphi)$.
\item The maximal minimal $F$-crystal $(M_-, \varphi)$ inside $(M[1/p], \varphi[1/p])$ contained in $(M, \varphi)$, i.e. for every minimal $F$-crystal $(M', \varphi)$ inside $(M[1/p], \varphi[1/p])$ contained in $(M, \varphi)$, we have $(M', \varphi) \subset (M_-, \varphi)$.
\end{enumerate}
\end{proposition}
\begin{proof}
This is a generalization of \cite[Proposition 5.19]{Vasiu:traversosolved} and we use the same strategy to prove this proposition.

We first prove the proposition in the isoclinic case. Suppose $(M, \varphi)$ is isoclinic, there exists a smallest $F$-valuation $w$ of slope $\lambda$ such that $w(M) \geq 0$ by Proposition \ref{prop:valuationminimal}. Let $M_+ = M[1/p]^{w \geq 0}$. By Theorem \ref{thm:anotherminimaldes}, any other minimal $F$-crystal $(M', \varphi)$ inside $(M[1/p], \varphi[1/p])$ containing $(M, \varphi)$ is of the form $M[1/p]^{w' \geq 0}$ such that $w \leq w'$. Thus $(M_+, \varphi) \subset (M', \varphi)$. This completes the proof of Part (1) in the isoclinic case.

Now we prove the general case of Part (1). Let $(M[1/p], \varphi[1/p]) \cong \bigoplus_{i=1}^t (N_i, \varphi)$ where each $(N_i, \varphi)$ is isoclinic of $\lambda_i$ such that $\lambda_i \neq \lambda_j$ if $i \neq j$. Let $M_i$ be the image of $M$ in $N_i$. We claim that $(M_+, \varphi) \cong \bigoplus_{i=1}^t ((M_i)_+, \varphi)$. Since $(M, \varphi) \subset \bigoplus_{i=1}^t (M_i, \varphi)$, we know that $(M, \varphi) \subset \bigoplus_{i=1}^t ((M_i)_+, \varphi)$. Each minimal $F$-crystal $(M', \varphi)$ containing $(M, \varphi)$ is a direct sum $\bigoplus_{i=1}^t(M'_i, \varphi)$ where each $M'_i$ contains $M_i$. Thus $((M_i)_+, \varphi) \subset (M'_i, \varphi)$ and $(M_+, \varphi) = \bigoplus_{i=1}^t ((M_i)_+, \varphi)$.

To prove the existence of the maximal minimal $F$-crystal containing $(M, \varphi)$ in the isoclinic case, we use duality. Recall that for any $F$-crystal $\mathcal{M} = (M, \varphi)$ we denote by $\mathcal{M}^* = (M^*, \varphi^*)$ its dual where $M^* = \textrm{Hom}_{W(k)}(M, W(k))$ and $\varphi^* (f) = \sigma \circ f \circ \varphi^{-1}$ for $f \in \textrm{Hom}_{W(k)}(M, W(k))$. Note that $\mathcal{M}^*$ in general is only a latticed $F$-isocrystal and is not an $F$-crystal unless all Hodge slopes of $\mathcal{M}$ are zero. 

Let $e_r$ be the maximal Hodge slope of $\mathcal{M}$. Then we know that $(M^*, p^{e_r}\varphi^*)$ is an $F$-crystal. Let $((M^*)_+, p^{e_r}\varphi^*)$ be the minimal minimal $F$-crystal containing $(M^*, p^{e_r}\varphi^*)$, then $(((M^*)_+)^*, \varphi)$ is an $F$-crystal contained in $(M, \varphi)$ and it is minimal because its isomorphism number equals to the isomorphism number of $((M^*)_+, p^{e_r}\varphi^*)$. We claim that $(((M^*)_+)^*, \varphi)$ is the maximal minimal $F$-crystal contained in $(M, \varphi)$. If not, then there is a minimal $F$-crystal $(M', \varphi)$ such that $(((M^*)_+)^*, \varphi) \subsetneq (M', \varphi) \subset (M, \varphi)$. By taking duals, we have $(M^*, p^{e_r}\varphi^*) \subset (M'^*, p^{e_r}\varphi^*) \subsetneq ((M^*)_+, p^{e_r}\varphi^*)$ and this is a contradiction! Hence $(((M^*)_+)^*, \varphi) = (M_-, \varphi)$. This completes the proof of Part (2).
\end{proof}

For any finitely generated torsion $W(k)$-module $M$, we say that $m$ is the \emph{$p$-exponent} if $m$ is the smallest non-negative integer such that $p^mM = 0$. The following lemma is a generalization of {\cite[Lemma 5.22]{Vasiu:traversosolved}}. We will use the same strategy to prove it.

\begin{lemma} \label{lemma:pexp}
Let $(M, \varphi)$ be an $F$-crystal. The following three $W(k)$-module $M_+/M_-$, $M/M_-$, $M_+/M$ have the same $p$-exponents.
\end{lemma}
\begin{proof}
If $p^m$ annihilates $M_+/M$, then $p^mM_+ \subset M$. Since $(p^mM_+, \varphi)$ is minimal, we know that $p^mM_+ \subset M_-$ and thus $p^m$ annihilates $M_+/M_-$. Hence the $p$-exponent of $M_+/M$ and $M_+/M_-$ are the same. 

If $p^m$ annihilates $M/M_-$, then $p^mM \subset M_-$. Since $(p^{-m}M_{-},\varphi)$ is minimal, we know that $M_+ \subset p^{-m}M_-$ and thus $p^m$ annihilates $M_+/M_-$. Hence the $p$-exponent of $M/M_-$ and $M_+/M_-$ are the same.
\end{proof}

\section{Minimal heights}

Let $\mathcal{M}$ and $\mathcal{M}'$ be two isogenous non-ordinary $F$-crystals that are not necessarily isoclinic. Let $q = q_{\mathcal{M'}, \mathcal{M}}$ be the smallest non-negative integer such that $p^q$ annihilates $M/M'$ for some isogeny $\mathcal{M}' \hookrightarrow \mathcal{M}$. It is easy to check that $q_{\mathcal{M}',\mathcal{M}} = q_{\mathcal{M},\mathcal{M}'}$. See \cite[Proposition 1.4.4]{Vasiu:reconstructing} for a version of Dieudonn\'e modules, it essentially suggests the proof the following proposition.

\begin{proposition} \label{propositon:minimalheightestimate}
Let $\mathcal{M}$ and $\mathcal{M}'$ be two isogenous non-ordinary $F$-crystals that are not necessarily isoclinic. We have the following inequality: 
\[n_{\mathcal{M}} \leq n_{\mathcal{M}'} + 2q_{\mathcal{M'}, \mathcal{M}}.\]
\end{proposition}
\begin{proof}
To ease notation, let $q = q_{\mathcal{M'}, \mathcal{M}}$. Let $O_{\mathcal{M}}$ and $O_{\mathcal{M}'}$ be the level modules of $\mathcal{M}$ and $\mathcal{M}'$ respectively. As $p^qM \subset M' \subset M$, we have $p^{2q} \textrm{End}(M) \subset p^q \textrm{End}(M') \subset \textrm{End}(M)$. We want to show that $p^q O_{\mathcal{M}'} \subset O_{\mathcal{M}}$. For $i \in \mathbb{Z}_{\geq 0}$, we have 
\[\varphi^i(p^qO^+_{\mathcal{M}'}) \subset p^qO^+_{\mathcal{M}'} \subset p^q \textrm{End}(M') \cap L^+ \subset \textrm{End}(M) \cap L^+,\]
thus $p^qO^+_{\mathcal{M}'} \subset O^+_{\mathcal{M}}$. Similarly, $p^qO^0_{\mathcal{M}'} \subset O^0_{\mathcal{M}}$ and $p^qO^-_{\mathcal{M}'} \subset O^-_{\mathcal{M}}$. Hence $p^qO_{\mathcal{M}'} \subset O_{\mathcal{M}}$. Then
\[p^{2q+\ell_{\mathcal{M}'}} \textrm{End}(M) \subset p^{q+\ell_{\mathcal{M}'}} \textrm{End}(M') \subset p^{q}O_{\mathcal{M}'} \subset O_{\mathcal{M}} \subset \textrm{End}(M).\]
Thus $\ell_{\mathcal{M}} \leq \ell_{\mathcal{M}'} + 2q$. As $n_{\mathcal{M}} = \ell_{\mathcal{M}}$ and $n_{\mathcal{M}'} = \ell_{\mathcal{M}'}$ by \cite[Theorem 1.2]{Xiao:vasiuconjecture}, we conclude the proof of the proposition.
\end{proof}

\begin{definition}
Let $\mathcal{M}'$ be the minimal $F$-crystal isogenous to $\mathcal{M}$. We call $q_{\mathcal{M}} := q_{\mathcal{M'},\mathcal{M}}$ the \emph{minimal height} of $\mathcal{M}$. 
\end{definition}

\begin{corollary} \label{cor:estimateq}
For any $F$-crystal $\mathcal{M}$, we have $n_{\mathcal{M}} \leq 1 + 2q_{\mathcal{M}}$.
\end{corollary}
\begin{proof}
If $\mathcal{M}$ is ordinary, then $n_{\mathcal{M}} \leq 1$ and thus $n_{\mathcal{M}} \leq 1 + 2q_{\mathcal{M}}$ is always true. If $\mathcal{M}$ is non-ordinary, then the corollary follows from Proposition \ref{propositon:minimalheightestimate} and the inequality $n_{\mathcal{M'}} \leq 1$.
\end{proof}

Corollary \ref{cor:estimateq} suggests a method to estimate $n_\mathcal{M}$ via $q_{\mathcal{M}}$. On the other hand, the estimate of Corollary \ref{cor:estimateq} is optimal in the sense that there exists an $F$-crystal $\mathcal{M}$ such that $n_{\mathcal{M}} = 1+ 2q_{\mathcal{M}}$. When $\mathcal{M}$ is minimal with at least two distinct Hodge slopes, we will have $n_{\mathcal{M}} = 1 + 2q_{\mathcal{M}} = 1$. But on the other hand there also exists $F$-crystal $\mathcal{M}$ such that $n_{\mathcal{M}} < 1 + 2q_{\mathcal{M}}$. See \cite[Example 4.7.1]{Vasiu:reconstructing} for an example in the case of Dieudonn\'e modules.

\section{Estimates} \label{section:estimate}

From now on, we assume that $\mathcal{M} = (M, \varphi)$ is an isosimple $F$-crystal. In order to compute $q_{\mathcal{M}}$, it is enough to compute the $p$-exponent of $M_+/M_-$ by Lemma \ref{lemma:pexp}. Let $\lambda =  s/r$ be the Newton slope of $\mathcal{M}$ in the reduced form, we know that $(M[1/p], \varphi[1/p]) = K_{\lambda} \otimes_{B(\mathbb{F}_{p^r})} B(k)$. To define an $F$-valuation on $(M[1/p], \varphi[1/p])$ of slope $\lambda$, we want to have a unique way to express elements in $M[1/p]$.

\begin{lemma} \label{lemma:valuationonisocrystal}
Let $T \subset W(k)$ be the image of the Teichm\"uller lift $t : k \to W(k)$. Any element $x \in K_{\lambda} \otimes_{B(\mathbb{F}_{p^r})} B(k)$ can be uniquely expressed in the form
\[x = \sum_{i \geq n}^{\infty} \xi^i \otimes x_i, \]
with $x_i \in T$ for all $i \geq n$ where  $n \in \mathbb{Z}$ depends on $x$. The leading coefficient $x_n \neq 0$ when $x \neq 0$.
\end{lemma}
\begin{proof}
We first prove the lemma for simple tensors. Let $x = (\sum_{i=0}^{r-1} \xi^ia_i) \otimes b$ be a simple tensor in $K_{\lambda} \otimes_{B(\mathbb{F}_{p^r})} B(k)$ where $b \in B(k)$ and $a_i \in B(\mathbb{F}_{p^{r}})$. We have $x = \sum_{i=0}^{r-1} (\xi^ia_i \otimes b) = \sum_{i=0}^{r-1}(\xi^i \otimes a_ib)$. Let $b_i :=a_ib$ for all $0 \leq i \leq r-1$, we have $x = \sum_{i=0}^{r-1} (\xi^i  \otimes b_i)$. For each $i$, we know that $b_i = \sum_{j=n_i}^{\infty} c_{ij}p^j$ where $c_{ij} \in T$ and since $p = \xi^r$, we have $\xi^i \otimes b_i = \xi^i \otimes \sum_{j=n_i}^{\infty}c_{ij}p^j  = \sum_{j=n_i}^{\infty} (\xi^{i+jr} \otimes c_{ij})$. Therefore
\[x = \sum_{i=1}^{r-1}\sum_{j=n_i}^{\infty}(\xi^{i+jr} \otimes c_{ij}).\]
For any $0 \leq i_1 \neq i_2 \leq r-1$, we have $i_1+jr \neq i_2+j'r$ for any $j, j' \in \mathbb{Z}$. Hence no two $\xi^{i+jr} \otimes c_{ij}$ in the above expression share the same power of $\xi$. Therefore we get the desired form for $x$ where $n = \textrm{min}_{0 \leq i \leq r-1} \{i+n_ir\}$.

Any element of $K_{\lambda} \otimes_{B(\mathbb{F}_{p^r})} B(k) $ is a finite sum of simple tensors. Suppose 
\[x = \sum_{j=1}^t \left( \sum_{i \geq n_j} \xi^i \otimes x_{ij} \right).\]
Let $n = \textrm{min}_{1 \leq j \leq t}\{n_j\}$. Then
\[x = \sum_{i \geq n} \left( \xi^i \otimes (\sum_{j=1}^t x_{ij}) \right)= \sum_{i \geq n} \xi^i \otimes y_i,\]
where $y_i := \sum_{j=1}^t x_{ij}$ are not necessarily in $T$ as $t$ is not additive. Let $y_{n} = \sum_{l \geq 0} c_{l,n}p^l$ where $c_{l,n} \in T$. Then
\[x = \xi^{n} \otimes c_{0,n} +  \sum_{i \geq n+1} \xi^i \otimes y_i',\]
where $y_i'$ are not necessarily in $T$ just as above. Repeat this process until infinity so we will get
\[x = \sum_{i \geq n} \xi^i \otimes x_i.\]

Suppose $x = \sum_{i \geq n} \xi^i \otimes x_i = \sum_{i \geq n'} \xi^i \otimes x'_i$. We want to show that $n=n'$ and $x_i=x_i'$ for all $i$. Using the fact that $\xi^r=p$, we have
\[\sum_{j=0}^{r-1} \left(\xi^j \otimes \sum_{i \equiv j \; \textrm{mod} \; r} x_ip^{[\frac{i}{r}]}\right) =  \sum_{j=0}^{r-1} \left(\xi^j \otimes \sum_{i \equiv j \; \textrm{mod} \; r} x'_ip^{[\frac{i}{r}]}\right).\]
As the set $\{\xi^j \otimes 1\}_{0 \leq j \leq r-1}$ is a basis of $K_{\lambda} \otimes_{B(\mathbb{F}_{p^r})} B(k)$, we have
\[\sum_{i \equiv j \; \textrm{mod} \; r} x_ip^{[\frac{i}{r}]} = \sum_{i \equiv j \; \textrm{mod} \; r} x'_ip^{[\frac{i}{r}]}\]
for all $0 \leq j \leq r-1$. As such expression is unique in $B(k)$, we get that $n=n'$ and $x_i = x'_i$ for all $i \geq n$.
\end{proof}

Lemma \ref{lemma:valuationonisocrystal} allows us to define an $F$-valuation $w$ of slope $\lambda$ on $(M[1/p], \varphi[1/p])$. For any $x = \sum_{i \geq n}^{\infty}\xi^i \otimes x_i \in M[1/p]$, where $x_n \neq 0$, define $w(x) = n/r$; if $x_n = 0$, that is, $x=0$, define $w(x) = \infty$. Clearly it is a valuation. We check that it is an $F$-valuation of slope $\lambda$. The Frobenius automorphism $\varphi$ acts on $K_{\lambda} \otimes_{B(\mathbb{F}_{p^r})} B(k)$ as multiplication on the left by $\xi^s$ in the first coordinate and $\sigma$ in the second coordinate, hence if $w(x) = n/r$, then $w(\varphi x) = (n+s)/r = w(x) + \lambda$. 

As $w(M[1/p]) = (1/r)\mathbb{Z}$ and $w(M)$ is bounded from below, let $n_0/r$ be the smallest element in $w(M)$. Rescale $w$ by making the translation of $n_0/r$, then $w$ is the smallest $F$-valuation of slope $\lambda$ such that $w(M) \geq 0$ and thus $M_+ = M[1/p]^{w \geq 0}$. To ease notation, we still use $w$ to denote the $F$-valuation after the rescaling. 

By Theorem \ref{thm:anotherminimaldes} and Proposition \ref{prop:valuationoffbyconstant}, we know that $M_- = M[1/p]^{w \geq \alpha}$, where $\alpha$ is the smallest element such that $M[1/p]^{w \geq \alpha} \subset M$. As $\alpha := m_{\alpha}/r \in (1/r)\mathbb{Z}$, we can also write $M_- = M[1/p]^{w > \alpha-1/r}$.  

Since $q_{\mathcal{M}}$ is the smallest element such that $p^{q_{\mathcal{M}}}M_+ \subset M_-$, we know that $q_{\mathcal{M}} = \lceil \alpha \rceil$ or equivalently $q_{\mathcal{M}} = \lfloor \alpha-1/r \rfloor +1$. Thus to find $q_{\mathcal{M}}$, it is enough to find a good estimate (upper bound) of $\alpha$ or equivalently a good estimate of $m_{\alpha}$. We prove a lemma that allows us to estimate $m$.

\begin{lemma} \label{lemma:lowerboundofm}
Let $m_1$ be an integer that satisfies the following property: for any integer $m_2 \geq m_1$, there exists some element $x \in M$ such that $w(x) = m_2/r$. Then for any $y \in M[1/p]$ with $w(y) \geq m_1/r$, we have $y \in M$.
\end{lemma}
\begin{proof}
Let $y \in M[1/p]$ be such that $w(y) = m_3/r \geq m_1/r$. Write $y = \sum_{i=m_3}^{\infty} \xi^i \otimes y_i$. By the condition of the lemma, we know that there exists $z \in M$ with $w(z) = m_3/r$. If $z = \sum_{i=m_3}^{\infty} \xi^i \otimes z_i$, then $w(y - zz^{-1}_{m_3}y_{m_3}) \geq m_3+1$. Note that $z^{-1}_{m_3}y_{m_3} \in W(k)$. Now repeat this process, we can write $y$ as an series of elements in $M$ with coefficients in $W(k)$, hence $y \in M$.
\end{proof}

\begin{corollary} 
The smallest $m_1$ satisfying the property in the Lemma \ref{lemma:lowerboundofm} is equal to $m_{\alpha}$.
\end{corollary}
\begin{proof}
This is by Lemma \ref{lemma:lowerboundofm} and the definition of $m_\alpha$.
\end{proof}

The following three operations map $M$ to itself.
\begin{enumerate}
\item The Frobenius map $\varphi$.
\item The Verschiebung map $\vartheta$, that is, the unique $\sigma^{-1}$-linear map such that $\varphi \vartheta = \vartheta \varphi = p^{e}$ where $e$ is the largest Hodge slope. 
\item Multiplication by $p$.
\end{enumerate}
For any nonzero $x = \sum_{i \geq n}\xi^i \otimes x_i \in M[1/p]$ with $x_n \neq 0$, we have
\begin{enumerate}
\item $w(\varphi x) = w(x) + \lambda = (n+s)/r$,
\item $w(\vartheta x) = w(x) + e - \lambda = (n+re-s)/r$.
\item $w(px) = w(x) + 1 = (n+r)/r$.
\end{enumerate}
As $w$ is the minimum $F$-valuation such that $w(M) \geq 0$, there exists an element $y \in M$ such that $w(y) = 0$. By applying the three operations mentioned above, we see that for any $n$ that is a  linear combination of $s, re-s$ and $r$ with coefficients in $\mathbb{Z}_{\geq 0}$, there exists an $x \in M$ such that $w(x) = n/r$. By Lemma \ref{lemma:lowerboundofm}, we know that $m_{\alpha} \leq g(s, re-s, r) + 1$ where $g(s, re-s, r)$ is the Frobenius number of $s, re-s$ and $r$, that is, it is the largest positive integer that cannot be expressed as a linear combination of $s, re-s$ and $r$ with coefficients in $\mathbb{Z}_{\geq 0}$. It exists as $\textrm{gcd}(s, re-s, r) = 1$.

Curtis \cite{Curtis:FN} showed that in some sense, a search for a simple closed formula of the Frobenius number of three or more natural numbers is impossible. There are many algorithms to find the Frobenius numbers of three or more natural numbers, for example see \cite{Denham:generating}. But our case is in a special situation and a closed formula for $g(s,re-s,r)$ can be found using the following theorem of Brauer and Shockley \cite{Brauer1962}.

\begin{theorem} \label{theorem:brauer}
Let $x, y$ and $z$ be three pairwise coprime positive integers. If $y+z \equiv 0$ modulo $x$, then the Frobenius number of $x$, $y$ and $z$ is
\[g(x,y,z) = \mathrm{max}\{\lfloor \frac{xz}{y+z} \rfloor y, \lfloor \frac{xy}{y+z} \rfloor z \} -x.\]
\end{theorem}

We are now ready to prove Theorem \ref{theorem:maintheoremb}.

\begin{proof}[Proof of Theorem \ref{theorem:maintheoremb}]
By the previous discussion, we know that $q_{\mathcal{M}} = \lfloor \alpha -1/r \rfloor + 1$. Since $\alpha = m_{\alpha}/r$ and $m_{\alpha} \leq g(s,re-s,r)+1$, we have \[q_{\mathcal{M}} \leq \lfloor \frac{g(s,re-s,r)}{r} \rfloor + 1.\]
By Corollary \ref{cor:estimateq}, we have
\[n_{\mathcal{M}} \leq 2q_{\mathcal{M}} + 1 \leq 2 \lfloor \frac{g(s,re-s,r)}{r} \rfloor +3. \tag{*}\]
By Theorem \ref{theorem:brauer}, let $x = r$, $y=s$ and $z=re-s$, we get
\[g(s,re-s,r) = \textrm{max}\{\lfloor \frac{re-s}{e} \rfloor s, \lfloor \frac{s}{e} \rfloor (re-s) \} -r.\]
Plug this into (*), we get the desired upper bound which proves Theorem \ref{theorem:maintheoremb}.
\end{proof}

We compare our new upper bound with some other upper bounds in \cite{Xiao:computing}, \cite{Vasiu:traversosolved} and \cite{Xiao:vasiuconjecture} in the following examples.

\begin{remark}
The upper bound of Theorem \ref{theorem:maintheoremb} can be restated using $\lambda$ as
\[n_{\mathcal{M}} \leq 2\mathrm{max}\{{\Big \lfloor}(r-\lceil\frac{s}{e}\rceil)\lambda {\Big \rfloor}, {\Big \lfloor} \lfloor \frac{s}{e}\rfloor (e- \lambda) {\Big \rfloor}\}+1.\]
\end{remark}

\begin{example} \label{example1}
Let $0,1,3$ be the Hodge slopes of an isosimple $F$-crystal $\mathcal{M}$. The Newton slope is $4/3$. From \cite[Theorem 1.2]{Xiao:computing}, we get that $n_{\mathcal{M}} \leq \lfloor 3 + (2-1)\frac{4}{3} \rfloor = 4$. By Theorem \ref{theorem:maintheoremb}, we get a slightly better upper bound $n_{\mathcal{M}} \leq 3$.
\end{example}

\begin{example} \label{example2}
Let $\mathcal{M}$ be an isosimple Dieudonn\'e module. We have $s = d$, $re-s = c$, and $r = c+d$ where $d$ and $c$ are the dimension and codimension of $\mathcal{M}$ respectively. Therefore $g(s, re-s, r) = g(d, c) = cd-c-d$. By Theorem \ref{theorem:maintheoremb}, we have $n_{\mathcal{M}} \leq 2\lfloor cd/(c+d) \rfloor +1$. This is almost as good as the optimal estimate which is $\lfloor 2cd/(c+d) \rfloor$; see \cite[Theorem 1.3]{Vasiu:traversosolved} and \cite[Corollary 3.5]{Xiao:computing}. The two estimates are equal if $cd/(c+d) - \lfloor cd/(c+d) \rfloor \geq 1/2$.
\end{example}

\begin{example} \label{example3}
Let $\mathcal{M}$ be an isosimple $F$-crystal of rank $2$. Then its Hodge slopes are $0, e$ where $e$ is an odd positive integer (otherwise $(M, \varphi)$ is not isosimple). Theorem \ref{theorem:maintheoremb} predicts that $n_{\mathcal{M}} \leq e$. This is optimal according to \cite[Theorem 6.1]{Xiao:vasiuconjecture}.
\end{example}

\section*{Acknowledgement}
The author would like to thank Adrian Vasiu for several suggestions during the preparation of this manuscript.

\bibliographystyle{amsplain}
\bibliography{references}

\end{document}